\documentclass[11pt]{amsart}
\usepackage{amsfonts}
\usepackage[a4paper,margin=1in]{geometry}
\usepackage[T1]{fontenc}
\usepackage[utf8]{inputenc}
\usepackage{lmodern}
\usepackage{amsmath,amssymb,amsthm,mathtools}
\usepackage{enumitem}
\usepackage{microtype}
\usepackage{hyperref}

\hypersetup{
  colorlinks=true,
  linkcolor=blue,
  citecolor=blue,
  urlcolor=blue,
  pdftitle={An Oskolkov-Zhizhiashvili Criterion for Rectangular Fourier Sums},
  pdfauthor={Ushangi Goginava}
}
\numberwithin{equation}{section}
\theoremstyle{plain}
\newtheorem{theorem}{Theorem}[section]
\newtheorem{proposition}[theorem]{Proposition}
\newtheorem{lemma}[theorem]{Lemma}
\newtheorem{corollary}[theorem]{Corollary}

\newtheorem{problem}[theorem]{Problem}
\newtheorem*{zh}{Theorem Zh}
\theoremstyle{remark}
\newtheorem{remark}[theorem]{Remark}

\begin{document}
\title[Oskolkov--Zhizhiashvili criterion]{An Oskolkov--Zhizhiashvili
Criterion for Rectangular Fourier Sums}
\author{Ushangi Goginava}
\address{Department of Mathematical Sciences\\
United Arab Emirates University, P.O. Box No. 15551\\
Al Ain, Abu Dhabi, United Arab Emirates}
\email{zazagoginava@gmail.com; ugoginava@uaeu.ac.ae}

\begin{abstract}
Let $S_{\mathbf n}f$ denote the symmetric rectangular partial sums of the
trigonometric Fourier series of a function on the $d$-dimensional torus.
We prove a summable endpoint criterion at the Zhizhiashvili critical scale
for all $d\ge2$ and $1\le p\le2$. The criterion allows a general summable
secondary weight at the iterated-logarithmic level and contains, as special
cases, a double-logarithmic endpoint criterion and an $L^p$ Oskolkov-type
corollary. In particular, it answers the Zhizhiashvili--Marcinkiewicz
problem for $1<p<2$ and sharpens Zhizhiashvili's classical sufficient
conditions in the endpoint cases $p=1$ and $p=2$. 
\end{abstract}

\maketitle

\section{Introduction}

The one-dimensional theory of pointwise convergence of Fourier series is
shaped by the Carleson--Hunt theorem: if $f\in L^p(\mathbb{T})$ and $%
1<p<\infty$, then the Fourier series of $f$ converges almost everywhere to $%
f $ \cite{Carleson1966,Hunt1968}. In several variables the situation is
essentially different, because there are many inequivalent ways to let the
set of frequencies tend to infinity. In this paper we study symmetric
rectangular partial sums 
\begin{equation*}
S_{\mathbf{n}}f(x)=\sum_{|k_1|\le n_1}\cdots\sum_{|k_d|\le n_d}\widehat f(%
\mathbf{k})e^{i\mathbf{k}\cdot x}, \qquad \mathbf{n}=(n_1,\ldots,n_d),
\end{equation*}
and convergence in the sense of Pringsheim, that is, 
\begin{equation*}
S_{\mathbf{n}}f(x)\to f(x) \qquad \text{as}\qquad \min_{1\le j\le
d}n_j\to\infty.
\end{equation*}

A fundamental warning was given by Fefferman \cite{Fefferman1971Divergence}. Unlike the one-dimensional
Carleson--Hunt theorem, continuity alone does not force Pringsheim
convergence of rectangular partial sums in dimension two. More precisely,
Fefferman constructed continuous functions of two variables whose double
Fourier series diverge in the rectangular sense; this result made clear that
multidimensional rectangular convergence requires additional quantitative
smoothness assumptions \cite{Fefferman1971Divergence,Dyachenko2002}. We
shall use the $L^p$ modulus of continuity 
\begin{equation*}
\omega(\delta,f)_p=\sup_{\|h\|_{\mathbb{T}^d}\le\delta}\left\|f(\cdot+h)-f(%
\cdot)\right\|_{L^p(\mathbb{T}^d)}
\end{equation*}
as the measure of that smoothness.

The systematic study of almost everywhere convergence of multiple Fourier
series under logarithmic moduli of continuity was initiated by
Zhizhiashvili \cite [Ch. 2]{Zhizhiashvili1996}. The following theorem is the starting point of the present
paper.

\begin{zh}[Zhizhiashvili]
\label{thm:Zhizhiashvili-intro} Let $d\ge2$. The following assertions hold.
\begin{enumerate}[label=\textup{(\roman*)},leftmargin=2.2em]
\item If $f\in L^1(\mathbb{T}^d)$ and, for some $\varepsilon>0$,
\begin{equation}
\omega(\delta,f)_1\le \frac{A(f)}{(\log(1/\delta))^{d+\varepsilon}}, \qquad
0<\delta\le \frac12,  \label{eq:Zh-p1}
\end{equation}
then the rectangular Fourier series of $f$ converges almost everywhere in
the sense of Pringsheim.

\item If $f\in L^2(\mathbb{T}^d)$ and, for some $\varepsilon>0$,
\begin{equation}
\omega(\delta,f)_2\le \frac{A(f)}{(\log(1/\delta))^{d/2+\varepsilon}},
\qquad 0<\delta\le \frac12,  \label{eq:Zh-p2}
\end{equation}
then the rectangular Fourier series of $f$ converges almost everywhere in
the sense of Pringsheim.
\end{enumerate}
\end{zh}

This is one of the central results in Zhizhiashvili's theory of multiple
trigonometric Fourier series; see \cite{Zhizhiashvili1975,Zhizhiashvili1996}%
. Zhizhiashvili explicitly asked how sharp the logarithmic powers in %
\eqref{eq:Zh-p1} and \eqref{eq:Zh-p2} are. In particular, the exact
endpoint, where the positive number $\varepsilon$ is removed, is a natural
multidimensional analogue of a classical Zygmund-type problem in one
dimension.

The sharpness picture is delicate. Bakhbukh and Nikishin proved that there
are continuous functions with logarithmic modulus of continuity of order $%
1/\log(1/\delta)$ for which the double Fourier series diverges everywhere \cite{BakhbukhNikishin1973,Dyachenko2002}. In even
dimensions Bakhvalov later constructed continuous examples with logarithmic
moduli matching the expected critical powers and with rectangular divergence
in the $\lambda$-restricted sense \cite{Bakhvalov1997,Bakhvalov2002}. These
results show that the endpoint cannot be treated as a routine consequence of
uniform continuity.

A second major benchmark is Oskolkov's two-dimensional refinement (see \cite{Oskolkov1974}). In one of
its standard forms, Oskolkov's theorem says that if $f\in C(\mathbb{T}^{2})$
and the uniform modulus satisfies 
\begin{equation*}
\omega (\delta ,f)_{\infty }\,=o\left( \frac{1}{\log (1/\delta )\,\log \log
\log (e^{e}/\delta )}\right) ,\qquad \delta \rightarrow 0+,
\end{equation*}%
then the double Fourier series converges almost everywhere in the
rectangular sense; see \cite{Oskolkov1974} and the discussion in \cite%
{Bakhvalov2002,Dyachenko2002}. This theorem is genuinely two-dimensional.
The corresponding sharp Oskolkov-type assertion for dimensions $d>2$ appears
to remain open. This distinction is important below: our $L^p$ consequence
is a proved $d$-dimensional sufficient criterion with a summable secondary
factor, not a solution of the sharp higher-dimensional Oskolkov problem.

There is also a Marcinkiewicz--Plessner line of ideas. In dimension one,
Plessner's $L^2$ theorem and Marcinkiewicz's extension to $L^p$, $p>1$, give
almost everywhere estimates for Fourier partial sums in terms of Fourier
coefficients and moduli of continuity; see \cite{Bari1961,Zygmund2002}.
Zhizhiashvili \cite [Ch. 2]{Zhizhiashvili1996} extended the Plessner part to several variables and used it in
the proof of Theorem Zh (ii). He then asked for a
multidimensional analogue of Marcinkiewicz's theorem. In the present setting
this leads to the following Zhizhiashvili--Marcinkiewicz problem (see \cite[Ch.~2]{Zhizhiashvili1996}).

\begin{problem}[Zhizhiashvili--Marcinkiewicz problem]
\label{prob:Zh-Marcinkiewicz} Let $d\ge2$ and $1<p<2$. Does the logarithmic
condition 
\begin{equation}
\omega(\delta,f)_p\le \frac{A(f)}{(\log(1/\delta))^{d/p+\varepsilon}},
\qquad 0<\delta\le\frac12,  \label{eq:Zh-Marcinkiewicz-condition}
\end{equation}
with arbitrary $\varepsilon>0$, imply almost everywhere Pringsheim
convergence of the rectangular Fourier series of every $f\in L^p(\mathbb{T}%
^d)$?
\end{problem}

The main theorem below answers Problem \ref{prob:Zh-Marcinkiewicz}. It is a
near-endpoint result at the dimensionally correct logarithmic power and also
includes the endpoint spaces $p=1$ and $p=2$. Thus the improvement is new
even in the cases covered by Zhizhiashvili's original $L^1$ and $L^2$
theorem.

For $0<\delta\le1/2$ set 
\begin{equation}
\Lambda_1(\delta)=\log(1/\delta),\qquad
\Lambda_2(\delta)=\log\log(e^e/\delta),\qquad
\Lambda_3(\delta)=\log\log\log(e^{e^e}/\delta).
\label{eq:Lambda-definitions}
\end{equation}
The constants inside the logarithms are used only to keep the expressions
positive on the whole interval $0<\delta\le1/2$.

\begin{theorem}[Main summable Zhizhiashvili endpoint theorem]
\label{thm:main} Let $d\ge2$, $1\le p\le2$, and let $\psi:[1,\infty)\to(0,%
\infty)$ be non-increasing and satisfy 
\begin{equation}
\sum_{k=1}^{\infty}\psi(k)<\infty .  \label{eq:psi-summable}
\end{equation}
Let $f\in L^p(\mathbb{T}^d)$. Assume that 
\begin{equation}
\omega(\delta,f)_p\le
A(f)\,\Lambda_1(\delta)^{-d/p}\,\psi(\Lambda_2(\delta)), \qquad
0<\delta\le\frac12 .  \label{eq:main-condition}
\end{equation}
Then the rectangular partial sums 
converge to $f$ almost everywhere in the sense of Pringsheim; that is, 
\begin{equation*}
S_{\mathbf{n}}f(x)\to f(x)\qquad \text{for almost every }x\in\mathbb{T}^d
\end{equation*}
as $\min_{1\le j\le d}n_j\to\infty$.
\end{theorem}

\begin{corollary}[Double-logarithmic endpoint criterion]
\label{cor:double-log} Let $d\ge2$, $1\le p\le2$, $\eta>0$, and let $f\in
L^p(\mathbb{T}^d)$. If 
\begin{equation}
\omega(\delta,f)_p\le \frac{A(f)}{(\log(1/\delta))^{d/p}(\log\log(e^e/%
\delta))^{1+\eta}}, \qquad 0<\delta\le\frac12,
\label{eq:double-log-condition}
\end{equation}
then $S_{\mathbf{n}}f(x)\to f(x)$ almost everywhere in the sense of
Pringsheim.
\end{corollary}

\begin{corollary}[$d$-dimensional $L^p$ Oskolkov-type consequence] 
\label{cor:L2-oskolkov} Let $d\ge2$, $\eta>0$, $1\le p\le2$, and let $f\in L^p(\mathbb{T}%
^d)$. If 
\begin{equation}
\omega(\delta,f)_p\le \frac{A(f)}{(\log(1/\delta))^{d/p}
\log\log(e^e/\delta) (\log\log\log(e^{e^e}/\delta))^{1+\eta}}, \qquad
0<\delta\le\frac12,  \label{eq:L2-oskolkov-condition}
\end{equation}
then the $d$-dimensional rectangular Fourier series of $f$ converges almost
everywhere in the sense of Pringsheim.
\end{corollary}

\begin{remark}[The Oskolkov comparison]
\label{rem:oskolkov-comparison}
Corollary \ref{cor:L2-oskolkov} may be viewed as a $d$-dimensional
Oskolkov-type analogue at the $L^p$ Zhizhiashvili scale. The comparison has
the following precise meaning. Oskolkov's classical theorem is a
uniform-norm theorem for continuous functions on $\mathbb{T}^2$, with a
different secondary logarithmic scale. Corollary \ref{cor:L2-oskolkov} is
instead an $L^p(\mathbb{T}^d)$ sufficient criterion valid in every dimension
$d\ge2$. Thus Corollary \ref{cor:L2-oskolkov} should not be read as a
solution of the sharp higher-dimensional Oskolkov problem.
\end{remark}

\begin{remark}
Theorem \ref{thm:main} is stronger than a solution of Problem \ref%
{prob:Zh-Marcinkiewicz}, because it replaces the arbitrary positive
logarithmic margin in \eqref{eq:Zh-Marcinkiewicz-condition} by the critical
power and an arbitrary summable iterated-logarithmic weight. Corollary \ref%
{cor:double-log} gives the simple double-logarithmic form, while Corollary %
\ref{cor:L2-oskolkov} gives the $L^{p}$ Oskolkov-type form emphasized above.
In particular, the theorem contains the endpoint cases $p=1$ and $p=2$ and
improves the sufficient conditions in Theorem Zh
from $(\log (1/\delta ))^{-d/p-\varepsilon }$ to the sharper critical scale
with summable secondary factors. It should be noted that the case $p=1$ of Corollary \ref{cor:L2-oskolkov} was established by Stokolos \cite{Stok}. For the case $p=2$, we also refer to \cite{Gog-On2020}, where almost everywhere convergence of multiple Fourier series with respect to a general orthonormal system was established.

\end{remark}

%The proof follows an approximation-and-maximal-operator strategy, but one
%technical point is crucial. One cannot obtain
%\begin{equation*}
%\left\|\mathcal{M}P\right\|_{p}\le C_{p,d}(1+r)^{d/p}\left\|P\right\|_{p}
%\end{equation*}
%by interpolating a maximal operator whose definition depends on the spectrum
%of $P$. Instead, for each fixed $r$ we introduce the genuine sublinear
%operator
%\begin{equation*}
%\mathcal{M}_r f(x)=\sup_{0\le n_j\le2^r}|S_{\mathbf{n}}f(x)|,
%\end{equation*}
%prove its $L^1$ and $L^2$ estimates, interpolate this fixed operator by the
%Marcinkiewicz interpolation theorem, and only then apply the result to
%trigonometric polynomials.

\section{Notation and rectangular partial sums}

Let 
\begin{equation*}
\mathbb{T}=\mathbb{R}/(2\pi\mathbb{Z}),\qquad \mathbb{T}^d=(\mathbb{R}/(2\pi%
\mathbb{Z}))^d,
\end{equation*}
equipped with normalized Haar measure $dm_d(x)=(2\pi)^{-d}\,dx$. For $t\in%
\mathbb{T}$ write 
\begin{equation*}
|t|_\mathbb{T}=\operatorname{dist}(t,2\pi\mathbb{Z})\in[0,\pi],
\end{equation*}
and for $h=(h_1,\ldots,h_d)\in\mathbb{T}^d$ put 
\begin{equation*}
\|h\|_{\mathbb{T}^d}=\max_{1\le j\le d}|h_j|_\mathbb{T}.
\end{equation*}

For $\mathbf{k}=(k_1,\ldots,k_d)\in\mathbb{Z}^d$ and $x=(x_1,\ldots,x_d)\in%
\mathbb{T}^d$, set 
\begin{equation*}
\mathbf{k}\cdot x=k_1x_1+\cdots+k_dx_d,\qquad e_{\mathbf{k}}(x)=e^{i\mathbf{k%
}\cdot x}.
\end{equation*}
If $f\in L^1(\mathbb{T}^d)$, its Fourier coefficients are 
\begin{equation*}
\widehat f(\mathbf{k})=\int_{\mathbb{T}^d}f(x)e^{-i\mathbf{k}\cdot
x}\,dm_d(x), \qquad \mathbf{k}\in\mathbb{Z}^d.
\end{equation*}
For $\mathbf{n}=(n_1,\ldots,n_d)\in\mathbb{N}_0^d$, the symmetric
rectangular partial sum is 
\begin{equation}
S_{\mathbf{n}}f(x)=\sum_{|k_1|\le n_1}\cdots\sum_{|k_d|\le n_d}\widehat f(%
\mathbf{k})e^{i\mathbf{k}\cdot x}.  \label{eq:rectangular-sum}
\end{equation}
Equivalently, 
\begin{equation*}
S_{\mathbf{n}}f=f*(D_{n_1}\otimes\cdots\otimes D_{n_d}),
\end{equation*}
where 
\begin{equation*}
D_n(t)=\sum_{|m|\le n}e^{imt}
\end{equation*}
is the one-dimensional Dirichlet kernel.

For $1\le p<\infty$, 
\begin{equation*}
\left\|f\right\|_{p}=\left(\int_{\mathbb{T}^d}|f(x)|^p\,dm_d(x)\right)^{1/p},
\end{equation*}
and for $p=\infty$ we use the essential supremum norm. The modulus of
continuity is 
\begin{equation}
\omega(\delta,f)_p=\sup_{\|h\|_{\mathbb{T}^d}\le\delta}\left\|f(\cdot+h)-f(%
\cdot)\right\|_{p}, \qquad \delta>0.  \label{eq:modulus}
\end{equation}
All logarithms are natural. Constants denoted by $C$ may change from line to
line; dependence on parameters is indicated by subscripts.

\section{de la Vall\'ee Poussin approximation}

We recall the product de la Vall\'ee Poussin means and the approximation
estimate needed in the proof. Let 
\begin{equation*}
F_N(t)=\sum_{|m|<N}\left(1-\frac{|m|}{N}\right)e^{imt},\qquad N\ge1,
\end{equation*}
be the Fej\'er kernel. With normalized measure, $F_N\ge0$, $\int_\mathbb{T }%
F_N\,dm_1=1$, and $\left\|F_N\right\|_{1}=1$. Define 
\begin{equation}
V_N(t)=2F_{2N}(t)-F_N(t),\qquad N\ge1.  \label{eq:vp-kernel}
\end{equation}
Its Fourier multiplier $v_N$ is given by 
\begin{equation}
v_N(m)= 
\begin{cases}
1, & |m|\le N, \\ 
2-|m|/N, & N<|m|<2N, \\ 
0, & |m|\ge2N.%
\end{cases}
\label{eq:vp-multiplier}
\end{equation}
Thus $V_N$ reproduces frequencies $|m|\le N$ and has spectrum in $(-2N,2N)$.
Also 
\begin{equation}
\left\|V_N\right\|_{L^1(\mathbb{T})}\le
2\left\|F_{2N}\right\|_{1}+\left\|F_N\right\|_{1}=3.  \label{eq:vp-l1}
\end{equation}

For $d\ge2$, put 
\begin{equation*}
V_N^{[d]}(x)=\prod_{j=1}^dV_N(x_j),
\end{equation*}
and define 
\begin{equation}
\mathcal{V}_N f=f*V_N^{[d]}.  \label{eq:vp-mean}
\end{equation}
Equivalently, 
\begin{equation}
\widehat{\mathcal{V}_N f}(\mathbf{k})=\left(\prod_{j=1}^d
v_N(k_j)\right)\widehat f(\mathbf{k}), \qquad \mathbf{k}\in\mathbb{Z}^d.
\label{eq:vp-fourier}
\end{equation}
The following elementary properties will be used repeatedly.

\begin{lemma}[de la Vall\'ee Poussin properties]
\label{lem:vp-properties} Let $1\le p\le\infty$ and $N\ge1$.

\begin{enumerate}[label=\textup{(V\arabic*)},leftmargin=2.8em]
\item $\mathcal{V}_N f$ has spectrum contained in $[-2N,2N]^d$.

\item If $Q$ is a trigonometric polynomial with spectrum contained in $%
[-N,N]^d$, then $\mathcal{V}_N Q=Q$.

\item $\left\|\mathcal{V}_N f\right\|_{p}\le 3^d\left\|f\right\|_{p}$.

\item If $1\le p<\infty$, then 
\begin{equation}
\left\|f-\mathcal{V}_N f\right\|_{p}\le C_{p,d}\,\omega(1/N,f)_p, \qquad
N\ge2.  \label{eq:vp-jackson}
\end{equation}
\end{enumerate}
\end{lemma}

\begin{proof}
The first two assertions follow directly from the multiplier formula %
\eqref{eq:vp-multiplier}. The third follows from \eqref{eq:vp-l1} and
Young's inequality, because $\left\|V_N^{[d]}\right\|_{1}\le3^d$.

For the last assertion, let $E_N(f)_p$ be the best $L^p$ approximation of $f$
by trigonometric polynomials with spectrum contained in $[-N,N]^d$. The
multidimensional Jackson inequality gives 
\begin{equation*}
E_N(f)_p\le C_{p,d}\omega(1/N,f)_p;
\end{equation*}
see, for example, \cite{DeVoreLorentz1993,Zygmund2002}. Choose $Q_N$ with $%
\left\|f-Q_N\right\|_{p}\le2E_N(f)_p$. Since $\mathcal{V}_N Q_N=Q_N$, 
\begin{equation*}
f-\mathcal{V}_N f=(f-Q_N)-\mathcal{V}_N(f-Q_N),
\end{equation*}
and the $L^p$ boundedness of $\mathcal{V}_N$ yields \eqref{eq:vp-jackson}.
\end{proof}

\section{Finite rectangular maximal estimates}

For $r\in\mathbb{N}_0$ define the fixed finite rectangular maximal operator 
\begin{equation}
\mathcal{M}_r f(x)=\sup_{0\le n_1,\ldots,n_d\le2^r}|S_{\mathbf{n}}f(x)|.
\label{eq:Mr-definition}
\end{equation}
This definition is independent of the spectrum of any polynomial to which it
will later be applied. 

We also write 
\begin{equation}
\mathcal{M}P(x)=\sup_{\mathbf{n}\in\mathbb{N}_0^d}|S_{\mathbf{n}}P(x)|
\label{eq:Mrect-polynomial}
\end{equation}
for a trigonometric polynomial $P$.

\begin{proposition}[Finite maximal estimate]
\label{prop:finite-max} Let $d\ge2$ and $r\in\mathbb{N}_0$.

\begin{enumerate}[label=\textup{(\roman*)},leftmargin=2.2em]
\item For every $f\in L^1(\mathbb{T}^d)$, 
\begin{equation}
\left\|\mathcal{M}_r f\right\|_{1}\le C_d(1+r)^d\left\|f\right\|_{1}.
\label{eq:Mr-L1}
\end{equation}

\item For every $f\in L^2(\mathbb{T}^d)$, 
\begin{equation}
\left\|\mathcal{M}_r f\right\|_{2}\le C_d(1+r)^{d/2}\left\|f\right\|_{2}.
\label{eq:Mr-L2}
\end{equation}

\item For $1<p<2$ and every $f\in L^p(\mathbb{T}^d)$, 
\begin{equation}
\left\|\mathcal{M}_r f\right\|_{p}\le C_{p,d}(1+r)^{d/p}\left\|f\right\|_{p}.
\label{eq:Mr-Lp}
\end{equation}
\end{enumerate}

Consequently, if $P$ is a trigonometric polynomial with $\operatorname{supp}\widehat
P\subset[-2^r,2^r]^d$, then 
\begin{equation}
\left\|\mathcal{M}P\right\|_{p}\le C_{p,d}(1+r)^{d/p}\left\|P\right\|_{p},
\qquad 1\le p\le2.  \label{eq:poly-max-p}
\end{equation}
\end{proposition}

\begin{proof}
We first prove the $L^1$ estimate. Put 
\begin{equation*}
D_r^*(t)=\sup_{0\le n\le2^r}|D_n(t)|.
\end{equation*}
The standard Dirichlet-kernel bound 
\begin{equation*}
|D_n(t)|\le C\min\{n+1,|t|_\mathbb{T}^{-1}\}
\end{equation*}
implies 
\begin{equation}
\left\|D_r^*\right\|_{L^1(\mathbb{T})}\le C(1+r).  \label{eq:Dstar-L1}
\end{equation}
For $0\le n_j\le2^r$, 
\begin{equation*}
|S_{\mathbf{n}}f|\le |f|*(D_r^*\otimes\cdots\otimes D_r^*).
\end{equation*}
Taking the supremum over such $\mathbf{n}$ and applying Young's inequality
gives \eqref{eq:Mr-L1}.

For $L^2$, let $Q_r f=S_{(2^r,\ldots,2^r)}f$. Then $S_{\mathbf{n}}f=S_{%
\mathbf{n}}Q_r f$ whenever $0\le n_j\le2^r$, and hence 
\begin{equation*}
\mathcal{M}_r f\le \sup_{\mathbf{n}\in\mathbb{N}_0^d}|S_{\mathbf{n}}Q_r f|.
\end{equation*}
We use the following product-logarithmic maximal inequality for rectangular
partial sums. In the one-sided form it is the maximal inequality appearing
in the proof of the Main Theorem of Chen \cite[Main Theorem, p.~549,
inequality (*)]{Chen1983}. In dimension $2$, a sharper predecessor was
proved by Sj\"olin \cite[Theorem 7.2, pp.~85--86]{Sjolin1971}; see also the
partial-summation argument in \cite[p.~4]{Sjolin1976}. The symmetric version
follows from the one-sided estimate by decomposing the frequency space into
the $2^d$ orthants. For every trigonometric polynomial $Q$ on $\mathbb T^d$,
these results give
\begin{equation}
\left\|\sup_{\mathbf{n}\in\mathbb{N}_0^d}|S_{\mathbf{n}}Q|\right\|_2^2 \le
C_d\sum_{\mathbf{k}\in\mathbb{Z}^d}|\widehat Q(\mathbf{k})|^2
\prod_{j=1}^d\log(2+|k_j|).  \label{eq:sjolin}
\end{equation}

%For completeness we indicate only the harmless reduction from the one-sided
%form to the symmetric sums used here. Decompose $\mathbb Z^d$ into the
%$2^d$ orthants
%\begin{equation*}
%E_\varepsilon=\{\mathbf k\in\mathbb Z^d:\, k_j\ge0\text{ if }\varepsilon_j=1,
%\ k_j<0\text{ if }\varepsilon_j=-1\},\qquad
%\varepsilon\in\{-1,1\}^d,
%\end{equation*}
%and write $Q=\sum_\varepsilon Q_\varepsilon$ according to this partition of
%its Fourier spectrum. After the reflection $y_j=\varepsilon_jx_j$, the
%symmetric partial sums of $Q_\varepsilon$ become one-sided rectangular
%partial sums of the polynomial with Fourier coefficients
%$\widehat Q(\mathbf k)$ indexed by $(|k_1|,\ldots,|k_d|)$. The product weight
%is unchanged: after reflection the one-sided index is $m_j=|k_j|$, and hence
%$\log(2+m_j)=\log(2+|k_j|)$. Applying Chen's one-sided estimate to each
%orthant and then Cauchy's inequality over the finite sum of $2^d$ orthants
%gives \eqref{eq:sjolin}.

Since $\operatorname{supp}\widehat{Q_r f}\subset[-2^r,2^r]^d$, the product of
logarithms in \eqref{eq:sjolin} is bounded by $C_d(1+r)^d$ on the support.
Parseval's identity therefore gives 
\begin{equation*}
\left\|\mathcal{M}_r f\right\|_{2}^2\le C_d(1+r)^d\left\|Q_r f\right\|_{2}^2
\le C_d(1+r)^d\left\|f\right\|_{2}^2,
\end{equation*}
which proves \eqref{eq:Mr-L2}.

For $1<p<2$, apply the Marcinkiewicz interpolation theorem to the fixed
sublinear operator $\mathcal{M}_r$. If $\theta\in(0,1)$ is determined by 
\begin{equation*}
\frac1p=(1-\theta)\cdot1+\theta\cdot\frac12,
\end{equation*}
then the interpolated power of $(1+r)$ is 
\begin{equation*}
d(1-\theta)+\frac d2\theta=\frac dp.
\end{equation*}
This proves \eqref{eq:Mr-Lp}. We cite \cite{SteinWeiss1971} for this
standard interpolation theorem for sublinear operators.

Finally suppose $P$ has spectrum in $[-2^r,2^r]^d$. If some $n_j>2^r$,
replacing $n_j$ by $2^r$ does not change the action of the $j$-th
one-dimensional frequency projection on $P$. Hence $\mathcal{M}P=\mathcal{M}%
_r P$. Estimate \eqref{eq:poly-max-p} follows from \eqref{eq:Mr-L1} when $%
p=1 $, from \eqref{eq:Mr-L2} when $p=2$, and from \eqref{eq:Mr-Lp} when $%
1<p<2$.
\end{proof}

\section{Proof of the main theorem and its consequences}

\begin{proof}[Proof of Theorem \protect\ref{thm:main}]
Let 
\begin{equation*}
r_k:=2^k, \qquad N_k:=2^{r_k-1}, \qquad k=1,2,\ldots,
\end{equation*}
Set 
\begin{equation*}
P_k=\mathcal{V}_{N_k}f, \qquad P_0=\mathcal{V}_1f.
\end{equation*}
By Lemma \ref{lem:vp-properties}, 
\begin{equation}
\left\|f-P_k\right\|_{p}\le C_{p,d}\omega(1/N_k,f)_p.  \label{eq:Pk-general}
\end{equation}
Since 
\begin{equation*}
\log N_k\asymp r_k, \qquad \Lambda_2(1/N_k)=\log\log(e^eN_k)\asymp
\log(e+r_k)\asymp k,
\end{equation*}
there exists a constant $c=c(d,p)>0$ such that, for all sufficiently large $%
k $, 
\begin{equation}
\left\|f-P_k\right\|_{p}\le C_{p,d}A(f)\,r_k^{-d/p}\,\psi(c k).
\label{eq:Pk-rate}
\end{equation}
 In particular, $%
P_k\to f$ in $L^p(\mathbb{T}^d)$.

Define the frequency blocks 
\begin{equation*}
g_0=P_0, \qquad g_k=P_k-P_{k-1},\quad k\ge1.
\end{equation*}
Then 
\begin{equation}
f=\sum_{k=0}^\infty g_k \qquad \text{in }L^p(\mathbb{T}^d).
\label{eq:Lp-decomposition}
\end{equation}
Moreover, by Lemma \ref{lem:vp-properties}, $g_k$ is a trigonometric
polynomial with spectrum contained in $[-2^{r_k},2^{r_k}]^d$.

Put 
\begin{equation*}
M_k(x)=\mathcal{M}g_k(x)=\sup_{\mathbf{n}\in\mathbb{N}_0^d}|S_{\mathbf{n}%
}g_k(x)|.
\end{equation*}
For every $1\le p\le2$, Proposition \ref{prop:finite-max} gives 
\begin{equation}
\left\|M_k\right\|_{p}\le C_{p,d}(1+r_k)^{d/p}\left\|g_k\right\|_{p}.
\label{eq:Mk-basic}
\end{equation}
On the other hand, \eqref{eq:Pk-rate} yields, for $k\ge2$, 
\begin{align}
\left\|g_k\right\|_{p} &\le
\left\|f-P_k\right\|_{p}+\left\|f-P_{k-1}\right\|_{p}  \notag \\
&\le C_{p,d}A(f)\,r_{k-1}^{-d/p}\,\psi(c(k-1)).  \label{eq:gk-rate}
\end{align}
Combining \eqref{eq:Mk-basic} and \eqref{eq:gk-rate}, and using $%
r_k=2r_{k-1} $, we obtain 
\begin{equation}
\left\|M_k\right\|_{p}\le C_{p,d}A(f)\,\psi(c(k-1)), \qquad k\ge2.
\label{eq:Mk-final}
\end{equation}
Because $\psi$ is non-increasing and \eqref{eq:psi-summable} holds, 
\begin{equation}
\sum_{k=2}^{\infty}\psi(c(k-1))<\infty.  \label{eq:psi-scaled-summable}
\end{equation}
For example, each interval $[m/c,(m+1)/c)$ contains at most a bounded number
of integers. Hence 
\begin{equation}
\sum_{k=1}^\infty \left\|M_k\right\|_{p}<\infty.  \label{eq:sum-Mk}
\end{equation}
Since the Haar measure of $\mathbb{T}^d$ is normalized and $p\ge1$, this
implies 
\begin{equation*}
\sum_{k=1}^\infty \left\|M_k\right\|_{1}<\infty.
\end{equation*}
By monotone convergence, 
\begin{equation}
\sum_{k=1}^\infty M_k(x)<\infty  \label{eq:pointwise-Mk-sum}
\end{equation}
for almost every $x\in\mathbb{T}^d$. Let $E_0$ be the set of points for
which \eqref{eq:pointwise-Mk-sum} holds.

For $x\in E_0$, the numerical series $\sum_{k\ge1}g_k(x)$ converges
absolutely, because $|g_k(x)|\le M_k(x)$. Define 
\begin{equation*}
F(x)=\sum_{k=0}^\infty g_k(x),\qquad x\in E_0.
\end{equation*}
Since $P_K=\sum_{k=0}^K g_k$ and $P_K\to f$ in $L^p$, we have $P_K\to f$ in
measure. On the other hand, $P_K(x)\to F(x)$ for every $x\in E_0$, and hence 
$P_K\to F$ in measure on $E_0$. By uniqueness of limits in measure, $F=f$
almost everywhere on $E_0$. Therefore there is a set $E\subset E_0$ of full
measure such that 
\begin{equation}
f(x)=\sum_{k=0}^\infty g_k(x),\qquad x\in E.  \label{eq:pointwise-series-f}
\end{equation}

Fix $x\in E$. For $K\ge1$, put 
\begin{equation*}
R_K=f-P_K.
\end{equation*}
By \eqref{eq:pointwise-series-f}, 
\begin{equation}
|R_K(x)|\le\sum_{k=K+1}^\infty M_k(x).  \label{eq:RK-tail}
\end{equation}
The polynomial $P_K$ has spectrum contained in $[-2^{r_K},2^{r_K}]^d$, and
therefore 
\begin{equation}
S_{\mathbf{n}}P_K=P_K \qquad \text{whenever } n_j\ge2^{r_K}\text{ for all }j.
\label{eq:PK-stable}
\end{equation}
For such $\mathbf{n}$, 
\begin{equation}
S_{\mathbf{n}}f(x)-f(x)=S_{\mathbf{n}}R_K(x)-R_K(x).
\label{eq:error-reduction}
\end{equation}

We estimate $S_{\mathbf{n}}R_K(x)$. The multiplier identity %
\eqref{eq:vp-fourier} shows that the Fourier coefficients of $%
g_k=P_k-P_{k-1} $ vanish on the cube $[-N_{k-1},N_{k-1}]^d$. For fixed $%
\mathbf{n}$, this implies $S_{\mathbf{n}}g_k=0$ for all sufficiently large $k
$. Since the tail $\sum_{k>K}g_k$ converges to $R_K$ in $L^p$ and the
operator $S_{\mathbf{n}}$ is bounded on $L^p$ for fixed $\mathbf{n}$, we get 
\begin{equation*}
S_{\mathbf{n}}R_K(x)=\sum_{k=K+1}^\infty S_{\mathbf{n}}g_k(x),
\end{equation*}
where only finitely many terms are nonzero. Hence 
\begin{equation}
|S_{\mathbf{n}}R_K(x)|\le\sum_{k=K+1}^\infty M_k(x).  \label{eq:SnRK-tail}
\end{equation}
Combining \eqref{eq:RK-tail}, \eqref{eq:error-reduction}, and %
\eqref{eq:SnRK-tail}, we obtain 
\begin{equation*}
|S_{\mathbf{n}}f(x)-f(x)|\le2\sum_{k=K+1}^\infty M_k(x)
\end{equation*}
whenever $n_j\ge2^{r_K}$ for every $j$. Taking the limit superior as $%
\min_jn_j\to\infty$ gives 
\begin{equation*}
\limsup_{\min_jn_j\to\infty}|S_{\mathbf{n}}f(x)-f(x)|
\le2\sum_{k=K+1}^\infty M_k(x).
\end{equation*}
Finally let $K\to\infty$. The right-hand side tends to zero by %
\eqref{eq:pointwise-Mk-sum}. Thus $S_{\mathbf{n}}f(x)\to f(x)$ for every $%
x\in E$, and $E$ has full measure.
\end{proof}

\begin{proof}[Proof of Corollary \protect\ref{cor:double-log}]
Apply Theorem \ref{thm:main} with $\psi(t)=t^{-1-\eta}$. The series $%
\sum_{k=1}^\infty k^{-1-\eta}$ converges, so the corollary follows.
\end{proof}

\begin{proof}[Proof of Corollary \protect\ref{cor:L2-oskolkov}]
Apply Theorem \ref{thm:main} with the given $p$ and
\begin{equation*}
\psi(t)=\frac{1}{t(\log(e+t))^{1+\eta}},\qquad t\ge1.
\end{equation*}
This function is non-increasing and $\sum_{k=2}^\infty 1/(k(\log
k)^{1+\eta})<\infty$. Moreover, 
\begin{equation*}
\log(e+\Lambda_2(\delta))\asymp \Lambda_3(\delta), \qquad 0<\delta\le\frac12.
\end{equation*}
Thus the hypothesis \eqref{eq:L2-oskolkov-condition} implies %
\eqref{eq:main-condition} with this choice of $\psi$, after changing the
constant $A(f)$. The conclusion follows from Theorem \ref{thm:main}.
\end{proof}

\section{Open question}
\label{sec:open-question}

We close with the natural Oskolkov--Zhizhiashvili endpoint problem suggested
by Theorem \ref{thm:main}. The exponent $d/p$ below is the same
dimensionally critical exponent as in \eqref{eq:main-condition}.

\begin{problem}[Oskolkov--Zhizhiashvili $L^p$ endpoint problem]
	\label{prob:oskolkov-lp}
	Let $d\ge2$, $1\le p\le2$, and let $f\in L^p(\mathbb{T}^d)$. Does the
	little-$o$ condition
	\begin{equation}
		\omega(\delta,f)_p=o\left(
		\frac{1}{(\log(1/\delta))^{d/p}
			\log\log\log(e^{e^e}/\delta)}
		\right),\qquad \delta\to0+,
		\label{eq:oskolkov-lp-little-o}
	\end{equation}
	imply almost everywhere Pringsheim convergence of the rectangular Fourier
	sums of $f$?
\end{problem}

Theorem \ref{thm:main} proves convergence under stronger summable secondary
assumptions. For instance, choosing
$\psi(t)=1/(t(\log(e+t))^{1+\eta})$ in \eqref{eq:main-condition} gives the
sufficient condition
\begin{equation*}
	\omega(\delta,f)_p=O\left(
	\frac{1}{\Lambda_1(\delta)^{d/p}\Lambda_2(\delta)
		\Lambda_3(\delta)^{1+\eta}}
	\right),\qquad \eta>0.
\end{equation*}
Problem \ref{prob:oskolkov-lp} asks whether these summable secondary factors
can be replaced by the weaker Oskolkov-type triple logarithm
$\Lambda_3(\delta)$, while preserving the dimensionally correct critical
power $\Lambda_1(\delta)^{-d/p}$. Thus it lies between the summable criteria
proved here and the bare critical endpoint question with no secondary
logarithmic factor. Such a result would be new in the endpoint cases $p=1$
and $p=2$; for $d>2$ it would give a higher-dimensional $L^p$ analogue of
the two-dimensional Oskolkov phenomenon. The present
approximation--maximal method does not settle
\eqref{eq:oskolkov-lp-little-o}: with the logarithmic blocks used in the
proof, \eqref{eq:oskolkov-lp-little-o} gives only block bounds of order
$o(1/\log k)$, which need not be summable.

\section*{Conflicts of Interest}

The authors declare that they have no conflicts of interest.

\section*{Data Availability}

Not applicable.

\end{document}